\documentclass[a4paper, 11pt, reqno]{amsart}
\usepackage[top = 1in, bottom = 0.9in, left = 1in, right = 1in]{geometry}

\usepackage{amssymb,amsthm,amsmath,graphicx,xcolor,mathtools,enumerate,mathrsfs}
\usepackage{tabularx}
\usepackage[shortlabels]{enumitem} 
\usepackage[british]{babel}
\usepackage[utf8]{inputenc}

\allowdisplaybreaks

\usepackage[mathlines]{lineno}
\usepackage{etoolbox} 

\newcommand*\linenomathpatch[1]{%
	\expandafter\pretocmd\csname #1\endcsname {\linenomath}{}{}%
	\expandafter\pretocmd\csname #1*\endcsname{\linenomath}{}{}%
	\expandafter\apptocmd\csname end#1\endcsname {\endlinenomath}{}{}%
	\expandafter\apptocmd\csname end#1*\endcsname{\endlinenomath}{}{}%
}
\newcommand*\linenomathpatchAMS[1]{%
	\expandafter\pretocmd\csname #1\endcsname {\linenomathAMS}{}{}%
	\expandafter\pretocmd\csname #1*\endcsname{\linenomathAMS}{}{}%
	\expandafter\apptocmd\csname end#1\endcsname {\endlinenomath}{}{}%
	\expandafter\apptocmd\csname end#1*\endcsname{\endlinenomath}{}{}%
}

\expandafter\ifx\linenomath\linenomathWithnumbers
\let\linenomathAMS\linenomathWithnumbers
\patchcmd\linenomathAMS{\advance\postdisplaypenalty\linenopenalty}{}{}{}
\else
\let\linenomathAMS\linenomathNonumbers
\fi

\linenomathpatchAMS{gather}
\linenomathpatchAMS{multline}
\linenomathpatchAMS{align}
\linenomathpatchAMS{alignat}
\linenomathpatchAMS{flalign}
\linenomathpatch{equation}

\nolinenumbers

\usepackage{hyperref}
\hypersetup{colorlinks, linkcolor={red!50!black}, citecolor={green!50!black}, urlcolor={blue!50!black}}

\usepackage{caption}
\usepackage{subcaption}

\usepackage{cleveref}
\theoremstyle{plain}

\newtheorem{theorem}{Theorem}[section]
\crefname{theorem}{Theorem}{Theorems}

\crefname{proposition}{Proposition}{Propositions}

\newtheorem{corollary}[theorem]{Corollary}
\crefname{corollary}{Corollary}{Corollaries}

\newtheorem{lemma}[theorem]{Lemma}
\crefname{lemma}{Lemma}{Lemmas}

\newtheorem{conjecture}[theorem]{Conjecture}
\crefname{conjecture}{Conjecture}{Conjectures}

\crefname{problem}{Problem}{Problem}

\newtheorem{claim}[theorem]{Claim}
\crefname{claim}{Claim}{Claims}

\crefname{observation}{Observation}{Observations}

\crefname{setup}{Setup}{Setups}

\crefname{fact}{Fact}{Facts}

\crefname{algorithm}{Algorithm}{Algorithms}

\crefname{remark}{Remark}{Remarks}

\crefname{example}{Example}{Examples}

\theoremstyle{definition}

\crefname{definition}{Definition}{Definitions}

\crefname{construction}{Construction}{Constructions}

\crefname{question}{Question}{Questions}

\crefname{section}{Section}{Sections}
\crefname{figure}{Figure}{Figures}
\numberwithin{equation}{section}
\crefname{enumi}{}{}

\definecolor{VividOrange}{HTML}{F15918}
\definecolor{NeonBlue}{HTML}{1F51FF}

\DeclarePairedDelimiter\abs{\lvert}{\rvert}
\DeclarePairedDelimiter\Abs{\lVert}{\rVert}
\renewcommand{\epsilon}{\varepsilon}
\renewcommand{\subset}{\subseteq}

\def\leq{\leqslant}

\def\geq{\geqslant}
\let\isom\cong

\newcommand{\II}{\mathcal{I}}

\newcommand{\im}{\operatorname{im}}
\newcommand{\cls}{\operatorname{cls}}
\newcommand{\binc}[2]{\begin{pmatrix}#1 \\ #2\end{pmatrix}}
\newcommand{\NN}{\mathbb{N}}
\newcommand{\ZZ}{\mathbb{Z}}
\newcommand{\RR}{\mathbb{R}}
\newcommand{\PP}{\mathbb{P}}
\newcommand{\EE}{\mathcal{E}}
\newcommand{\PM}{\mathcal{P}}
\newcommand{\CM}{\mathcal{C}}

\newenvironment{proofclaim}[1][Proof of the claim]{\begin{proof}[#1]}{\end{proof}}

\usepackage{tikz}

\title[]{Independent sets\\in discrete tori of odd sidelength}
\date{\today}
\author[Arras and Joos]{Patrick Arras and Felix Joos}
\thanks{The research leading to these results was partially supported by the Deutsche \mbox{Forschungsgemeinschaft} (DFG, German Research Foundation) -- 428212407.}
\address{Universität Heidelberg,
	Institut für Informatik,
	Im Neuenheimer Feld 205,
	69120 Heidelberg, Germany}
\email{\{arras,joos\}@informatik.uni-heidelberg.de}

\begin{document}
\begin{abstract}
	It is a well known result due to Korshunov and Sapozhenko that the hypercube in $n$ dimensions has $(1 + o(1)) \cdot 2 \sqrt e \cdot 2^{2^{n-1}}$ independent sets.
	Jenssen and Keevash investigated in depth Cartesian powers of cycles of fixed even lengths far beyond counting independent sets.
	They wonder to which extent their results extend to cycles of odd length,
	where not even the easiest case, counting independent sets in Cartesian powers of the triangle, is known.
	In this paper, we make progress on their question by providing a lower bound, which we believe to be tight.
	We also obtain a less precise lower bound for the number of independent sets in Cartesian powers of arbitrary odd cycles and show how to approach this question both with the cluster expansion method as well as more directly with isoperimetric inequalities.
\end{abstract}

\maketitle
\thispagestyle{empty}
\vspace{-0.4cm}

\section{Introduction}

The hypercube $\ZZ_2^n$ is arguably among the most well-investigated graphs because it is one of the very few explicitly constructable graphs that are (very) sparse.
In this paper, we focus on counting independent sets.
Korshunov and Sapozhenko~\cite{KS83} showed that there are $(1+o(1)) \cdot 2 \sqrt e \cdot 2^{2^{n-1}}$ independent sets in the hypercube.
Observe that hypercubes are bipartite graphs where each partition class has $2^{n-1}$ vertices.
Hence there are $2\cdot 2^{2^{n-1}}-1$ independent sets that are a subset of one of the partition classes,
which already reveals the majority of all independent sets.
Roughly a $(\sqrt e -1)/\sqrt e$-fraction of all independent sets contain vertices from both partition classes.
Among those, essentially all contain only very few vertices from one class and many from the other.
This is due to the fact that selecting some vertices in one class excludes many more vertices (the neighbours of these vertices) in the other class from being in the independent set.
This structural fact plays a dominant role in essentially all considerations regarding the number of independent sets or colourings in Cartesian powers of graphs.
Estimating the number of neighbours of sets of vertices is another prominent topic on its own and is captured under the umbrella of \emph{vertex-isoperimetric inequalities}.

The problem of calculating the number of independent sets in $\ZZ_2^n$ has been revisited and extended by many researchers~\cite{BGL21,EG12,Gal11,JP20,JPP22,KP22,Par22};
in particular, there are also results regarding the number of proper $q$-colourings in $\ZZ_2^n$~\cite{Gal03,KP20}.
Most recently, Jenssen and Keevash~\cite{JK20} investigated the topic in great depth.
Instead of the hypercube only, 
they consider Cartesian powers of even cycles, where they treat the complete graph on two vertices as the cycle on two vertices.
Then their results contain the hypercube.
These graphs are usually known as $n$-dimensional discrete tori, which we denote as $\ZZ_m^n$ (if the base graph is a cycle of length $m$).
Their results include a way to calculate asymptotically sharp formulas for both the number of independent sets and the number of proper $q$-colourings in tori $\ZZ_m^n$ with $m$ even. 
To this end, they utilize the cluster expansion approach from statistical physics.
This method is well-established in the field and has been exploited to answer many similar questions, most recently in~\cite{CDF+22,JPP23}.

Jenssen and Keevash~\cite{JK20} ask whether their results extend to tori that stem from cycles of odd length.
They however note that not even the number of independent sets is asymptotically known for Cartesian powers of a triangle.
Here we make progress in answering their question and also point out why these cases may be much more complex than the previously investigated ones.

\begin{theorem}\label{thm: lowerboundK3}
There are at least $(1-o(1)) \cdot 3 \cdot 2^{n-1} \cdot 2^{3^{n-1}} \cdot \exp( (3/2)^{n-1} )$ independent sets in $\ZZ_3^n$.
\end{theorem}

The bound in \cref{thm: lowerboundK3} deserves some explanation.
There are $3 \cdot 2^{n-1}$ maximum independent sets (this is nontrivial; see \cref{lem: no further maxindepsets}) of order $3^{n-1}$.
Hence there are about $3 \cdot 2^{n-1} \cdot 2^{3^{n-1}}$ subsets of these sets (again, this is nontrivial, because these sets overlap; see \cref{lem: intersection maxindepsets}).
For the hypercube, the subsets of the two maximum independents sets give up to a factor of $\sqrt{e}$ the correct count.
For $\ZZ_3^n$, there is a correction term of (at least) $\exp( (3/2)^{n-1})$;
that is, for $\ZZ_3^n$ only a double exponentially small fraction of the independent sets are a subset of some maximum independent set.
We conjecture that the bound in \cref{thm: lowerboundK3} is asymptotically tight.

The method we use for the proof of \cref{thm: lowerboundK3} does not work for odd $m \geq 5$. Nonetheless, we are able to determine the number of maximum independent sets in $\ZZ_m^n$ for all odd $m$. Additionally, we provide an alternative approach that yields the following lower bound.
\begin{theorem}\label{thm: lowerboundKm}
There are at least $2^{\lfloor m/2 \rfloor m^{n-1}} \cdot \exp( (1-o(1)) (m/2)^{n-1} )$ independent sets in $\ZZ_m^n$ for $m \geq 3$ odd.
\end{theorem}

In order to arrive at a sensible conjecture of how higher order terms in this asymptotic behaviour might look like, we also adapt the cluster expansion approach to our setting and calculate some initial terms.

As indicated already before, 
a key tool in determining the number of independent sets in discrete tori are appropriate isoperimetric inequalities for independent sets.
We prove the following inequality, which appears to us to be of general interest.
For simplicity, consider the tori $\ZZ_3^n$.
As it turns out,
these graphs are $3$-colourable and each colour class of each $3$-colouring is of size $3^{n-1}$.
Denote the three colour classes (of some 3-colouring) by $\EE(0), \EE(1), \EE(2)$.
We consider $A\subseteq \EE(0)$ and give a lower bound for the number of neighbours of $A$ in $\EE(1)$ (or $\EE(2)$).
Similar considerations can also be made for $\ZZ_m^n$ with $m \geq 5$ odd.

To this end, define $\EE_m^n(p)$ as the set of all vertices $v = (v_1, \ldots, v_n)$ in $\ZZ_m^n$ that satisfy $\sum_{i = 1}^n v_i \equiv p$ mod $m$. Our isoperimetric inequalities then take the following form.

\begin{theorem}\label{thm: isoperimetry}
	For $m, n \in \NN$ with $m \geq 3$ odd and $\ell \coloneqq \lfloor m/2 \rfloor$, consider $p \in \ZZ_m$ and $A \subset \EE_m^n(p)$. Setting $\alpha \coloneqq \abs A / m^{n-1}$, we then have 
	\[
	\abs{N_n(A, \EE_m^n(p \pm 1))} 
	\geq \abs A \left( 1 + \frac {1-\alpha}{\sqrt{\ell^3mn}} \right)
	\,.
	\]
\end{theorem}

We remark that our inequalities are of the same type as needed for bipartite tori.
Unfortunately, for $\ZZ_3^n$ it seems crucial to prove an appropriate lower bound for the number of neighbours in $\EE(2)$ of an independent set $A\subseteq \EE(0) \cup \EE(1)$. Here, it is important that $A$ is independent in $\ZZ_3^n$, and incorporating this condition appears to us as a significant complication.
Finding such an inequality would be very desirable.

\medskip

The paper is structured as follows.
We first fix some further notation in \cref{sec:nota}.
In \cref{sec:lower-bounds} we prove \cref{thm: lowerboundK3,thm: lowerboundKm}. The more general approach via cluster expansion is introduced in \cref{sec:clusterexpansion}, and \cref{sec:isoperimetry} deals with the proof of \cref{thm: isoperimetry}.

\section{Preliminaries}\label{sec:nota}
Let $\II(G)$ denote the set of all \emph{independent} sets in the graph $G$, that is all $I \subset V(G)$ such that $E(G[I]) = \emptyset$. For $k \in \NN_0$, we also define $\II_k(G) \coloneqq \{ I \in \II(G) \colon \abs I = k \}$. Moreover, let $\II^*(G) \coloneqq \{ I \in \II(G) \colon \abs I \geq \abs {I'} \text{ for all } I' \in \II(G) \}$ be the set of all \emph{maximum} independent sets in $G$. For a vertex set $X \subset V(G)$, we write $N_G(X) \coloneqq \bigcup_{x \in X} N_G(x) \setminus X$ for the \emph{neighbourhood} of $X$. If $Y \subset V(G)$ is another set, then we denote the neighbourhood of $X$ in $Y$ as $N_G(X, Y) \coloneqq N_G(X) \cap Y$.

Let $[k] \coloneqq \{ 1, \ldots, k \}$. We are interested in $n$-dimensional discrete tori $\ZZ_m^n$ with odd sidelength $m$. These graphs can be defined by $V(\ZZ_m^n) \coloneqq \{ 0, 1, \ldots, m-1 \}^n$ and
\begin{align*}
	E(\ZZ_m^n) \coloneqq
	\{ uv \mid \text{there is } j \in [n] \text{ such that } u_j &= v_j \pm 1 \text{ mod } m \\
	\text{ and } u_i &= v_i \text{ for all } i \in [n] \setminus \{ j \} \}
	\,.
\end{align*}

All our arguments will eventually examine the limit as $n \to \infty$. 
For the sake of brevity, we utilize the Landau notation $o(h(n))$ to denote any function $j(n)$ satisfying $j(n) / h(n) \to 0$ as $n \to \infty$. 
This is mainly used in statements of the form $f(n) \geq (1 - o(1)) g(n)$ to express that $g$ is an asymptotic lower bound for $f$ or as $f(n) = (1 \pm o(1)) g(n)$ to say that both $f(n) \geq (1 - o(1)) g(n)$ and $f(n) \leq (1 + o(1)) g(n)$ hold. 
Similarly, we write $O(h(n))$ to denote any $j(n)$ with $\limsup_{n \to \infty} j(n) / h(n) < \infty$.

We will also make use of a standard probability tool that, depending on the context, is known as either the FKG inequality~\cite{FKG71} or the Harris inequality~\cite{Har60}. 
In its most basic form, it asserts the following.
\begin{lemma}\label{lem: FKG}
Let $k \in \NN$ and consider the probability space $(\{ 0, 1 \}^k, \mathcal F, \PP)$. 
Define the partial order $\leq$ on $\{ 0, 1 \}^k$ by $x' \leq x$ if $x'_i \leq x_i$ for all $i \in [k]$. 
An event $A \in \mathcal F$ is called \emph{decreasing} if $x' \leq x$ and $x \in A$ imply $x' \in A$. 
Then for any two decreasing events $A, B \in \mathcal F$, we have $\PP[A \cap B] \geq \PP[A] \PP[B]$.
\end{lemma}

\section{Lower bounds}\label{sec:lower-bounds}
In this section, we prove the asymptotic lower bounds on $\abs {\II(\ZZ_m^n)}$ in \cref{thm: lowerboundK3,thm: lowerboundKm}, where $\II(G)$ refers to the set of independent sets in a graph $G$. 
We begin with the more detailed estimate for $\abs {\II(\ZZ_3^n)}$ in \cref{thm: lowerboundK3}.
The principal idea behind our argument is a modification of the simplified approach used by Sapozhenko~\cite{Sap89} for proving the lower bound of $\abs {\II(\ZZ_2^n)} \geq (1-o(1)) \cdot 2 \sqrt e \cdot 2^{2^{n-1}}$ in the hypercube, as presented by Galvin~\cite{Gal19} in his expository note. 
Recall that the hypercube possesses a unique bipartition into two maximum independent sets. 
Now choose one of them as the \emph{majority side} $I$ of the independent set to be constructed and select a relatively small independent set $A$ of $k$ \emph{defect vertices} on the \emph{minority side} $\overline I$. Then combine $A$ with a relatively large set $B \subset I \setminus N_{\ZZ_2^n}(A, I)$ to form an independent set $A \cup B \in \II(\ZZ_2^n)$. Essentially, the argument then comes down to proving that this process produces all but a negligible fraction of independent sets in the hypercube.

We would like to follow a similar strategy for $\ZZ_m^n$ with $m$ odd. We again start by selecting a maximum independent set $I \in \II^*(\ZZ_m^n)$ as the majority side. For odd~$m$, however, there is more than two choices: We observe in \cref{lem: generate maxindepsets,lem: no further maxindepsets} that all maximum independent sets in $\ZZ_m^n$ look the same, partitioning $V(\ZZ_m^n)$ into $m$ partition classes and selecting $\ell \coloneqq \lfloor m/2 \rfloor$ pairwise non-adjacent ones of these classes. This part is actually true for all odd $m$ and thus is stated in full generality.

In order to leave the maximum number of vertices in $I \setminus N_{\ZZ_m^n}(A, I)$ as potential members of $B$, we would like to select the $k$ defect vertices of $A$ in a way that minimizes $\abs {N_{\ZZ_m^n}(A, I)}$. One quickly finds that since $I$ comprises $\ell$ of the $m$ partition classes, the minority side $\overline I$ must contain two adjacent classes. As the vertices in these classes have the least neighbours in $I$, choosing $A$ as an independent set in the graph $H_n \subset \ZZ_m^n$ induced by these two classes is optimal.

Finally, this is again combined with a set $B \subset I \setminus N_{\ZZ_m^n}(A, I)$ to form an independent set $A \cup B \in \II(\ZZ_m^n)$. When calculating the exact numbers, our lower bound suggests that for $m = 3$, the number of defects $k$ is essentially Poisson-distributed with parameter $\lambda = (m/2)^{n-1}$. Additionally, prescribing a minimum size for $B$ guarantees that starting out with different maximum independent sets $I$ leads to different independent sets $A \cup B$, so no set in $\II(\ZZ_m^n)$ is counted multiple times. We then obtain the desired asymptotic lower bound by summing over a sufficiently large range of $k$ around $\lambda$.

As our first step, \cref{lem: max size of maxindepset,lem: generate maxindepsets,lem: only disjoint if shifted,lem: no further maxindepsets} examine the size, number, and structure of maximum independent sets in $\ZZ_m^n$ for $m$ odd.

\begin{lemma} \label{lem: max size of maxindepset}
For $m, n \in \NN$ with $m \geq 2$ and $\ell \coloneqq \lfloor m/2 \rfloor$, every $I \in \II(\ZZ_m^n)$ satisfies $\abs I \leq m^{n-1} \ell$.
\end{lemma}

\begin{proof}
We proceed by induction on $n$. For $n = 1$, the statement is trivial.
Suppose that for some $n > 1$, every $I' \in \II(\ZZ_m^{n-1})$ satisfies $\abs {I'} \leq m^{n-2} \ell$, and let $I \in \II(\ZZ_m^n)$ be arbitrary. Partition $V(\ZZ_m^n)$ into $V_q \coloneqq \{ v \in V(\ZZ_m^n) \colon v_n = q \}$ for $q \in \ZZ_m$ and observe that each $\ZZ_m^n[V_q]$ is isomorphic to $\ZZ_m^{n-1}$. In particular, this partitions $I$ into $m$ subsets $I_q = I \cap V_q$, which inherit the independence of $I$ and are thus isomorphic to independent sets $I_q' \in \II(\ZZ_m^{n-1})$. By the induction hypothesis, these satisfy $\abs {I_q} = \abs {I_q'} \leq m^{n-2} \ell$ and as there are $m$ of them, we obtain $\abs I = \sum_{q \in \ZZ_m} \abs {I_q} \leq m^{n-1} \ell$.
\end{proof}

\begin{lemma} \label{lem: generate maxindepsets}
For $m, n \in \NN$ with $m \geq 3$ odd, the map $\iota_n \colon \ZZ_m \times \{ \pm 1 \}^{n-1} \to \II^*(\ZZ_m^n)$ defined by
\begin{align*}
	\iota_n(q, \epsilon_1, \ldots, {}&\epsilon_{n-1}) \\ \coloneqq \Bigg\lbrace &v \in V(\ZZ_m^n) \colon v_1 + \sum_{i = 1}^{n-1} \epsilon_i v_{i+1} \in \{ q, q+2, \ldots, q+m-3 \} \mod m \Bigg\rbrace
\end{align*}
is well-defined and injective.
\end{lemma}

\begin{proof}
We start by showing that $\iota_n$ is well-defined. For this, let $uv$ be an arbitrary edge in $\ZZ_m^n$. Note that this means that $u, v$ only differ in one coordinate, and only by 1 mod $m$. Therefore, the sums $u_1 + \sum_{i = 1}^{n-1} \epsilon_i u_{i+1}$ and $v_1 + \sum_{i = 1}^{n-1} \epsilon_i v_{i+1}$ only differ by 1 mod $m$ and as the values $q, q+2, \ldots, q+m-3$ have pairwise differences of at least 2 mod $m$, any $I \in \im(\iota_n)$ can contain at most one of $u, v$. Since $uv \in E(\ZZ_m^n)$ was arbitrary, this proves $\im(\iota_n) \subset \II(\ZZ_m^n)$.
	
To see that each set in $\im(\iota_n)$ is indeed a maximum independent set, partition $\ZZ_m^n$ into $V_w \coloneqq \{ v \in V(\ZZ_m^n) \colon (v_2, \ldots, v_n) = w \}$ for all $w \in V(\ZZ_m^{n-1})$. Observe that for any $(q, \epsilon_1, \ldots, \epsilon_{n-1}) \in \ZZ_m \times \{ \pm 1 \}^{n-1}$, all vertices $v \in V_w$ produce the same sum $\sum_{i = 1}^{n-1} \epsilon_i v_{i+1}$, so $v_1 + \sum_{i = 1}^{n-1} \epsilon_i v_{i+1}$ will cycle through all possible values $0, \ldots, m-1$ mod $m$ as $v$ cycles through $V_w$. Irrespective of $q$, exactly $\ell \coloneqq \lfloor m/2 \rfloor$ of these values belong to $\{ q, q+2, \ldots, q+m-3 \}$, which shows that $\abs {\iota_n(q, \epsilon_1, \ldots, \epsilon_{n-1}) \cap V_w} = \ell$ for all $w \in V(\ZZ_m^{n-1})$. It is now easy to see that
\[
	\abs {\iota_n(q, \epsilon_1, \ldots, \epsilon_{n-1})}
	= \sum_{w \in V(\ZZ_m^{n-1})} \abs {\iota_n(q, \epsilon_1, \ldots, \epsilon_{n-1}) \cap V_w}
	= m^{n-1} \ell
\]
and $\iota(q, \epsilon_1, \ldots, \epsilon_{n-1}) \in \II^*(\ZZ_m^n)$ follows by \cref{lem: max size of maxindepset}.
	
To see that $\iota_n$ is injective, let $(q, \epsilon_1, \ldots, \epsilon_{n-1}) \neq (q', \epsilon'_1, \ldots, \epsilon'_{n-1})$. If $q \neq q'$, then $q - q' \mod m \in \{ 1, \ldots, m-1 \}$. Without loss of generality, let $q - q' \mod m$ be odd, otherwise swap the roles of $(q, \epsilon_1, \ldots, \epsilon_{n-1})$ and $(q', \epsilon'_1, \ldots, \epsilon'_{n-1})$ and observe that $m$ being odd implies that $q' - q \equiv m - (q - q') \mod m \in \{ 1, \ldots, m-1 \}$ has a different parity than $q - q' \mod m$. Then $(q, 0, \ldots, 0) \in \iota_n(q, \epsilon_1, \ldots, \epsilon_{n-1})$, but $(q, 0, \ldots, 0) \notin \iota_n(q', \epsilon'_1, \ldots, \epsilon'_{n-1})$, and so $\iota_n(q, \epsilon_1, \ldots, \epsilon_{n-1}) \neq \iota_n(q', \epsilon'_1, \ldots, \epsilon'_{n-1})$.
	
If $q = q'$, we may assume by symmetry that $\epsilon_1 \neq \epsilon'_1$. Consider the vertex $v \coloneqq (q-1, 1, 0, \ldots, 0) \in V(\ZZ_m^n)$. Then $v_1 + \sum_{i = 1}^{n-1} \epsilon_i v_{i+1} = q-1 + \epsilon_1$, while $v_1 + \sum_{i = 1}^{n-1} \epsilon'_i v_{i+1} = q-1 + \epsilon'_1$. As $\epsilon_1 \neq \epsilon'_1 \in \{ \pm 1 \}$, one of these values is $q \in \{ q, q+2, \ldots, q+m-3 \}$ and the other is $q-2 \notin \{ q, q+2, \ldots, q+m-3 \}$. Consequently $v$ belongs to exactly one of the sets $\iota_n(q, \epsilon_1, \ldots, \epsilon_{n-1})$ and $\iota_n(q', \epsilon'_1, \ldots, \epsilon'_{n-1})$, which shows that $\iota_n(q, \epsilon_1, \ldots, \epsilon_{n-1}) \neq \iota_n(q', \epsilon'_1, \ldots, \epsilon'_{n-1})$.
\end{proof}

\begin{lemma} \label{lem: only disjoint if shifted}
For $m, n \in \NN$ with $m \geq 3$ odd, consider $x = (q, \epsilon_1, \ldots, \epsilon_{n-1}), x' = (q', \epsilon'_1, \ldots, \epsilon'_{n-1}) \in \ZZ_m \times (\pm 1)^{n-1}$. Then $\iota_n(x) \cap \iota_n(x') = \emptyset$ implies $q' \equiv q \pm 1 \mod m$ and $\epsilon_j = \epsilon_j'$ for all $j \in [n-1]$.
\end{lemma}

\begin{proof}
Suppose $\iota_n(x) \cap \iota_n(x') = \emptyset$. Without loss of generality, assume that $q - q' \mod m \in \{ 0, \ldots, m-1 \}$ is even, otherwise swap the roles of $x$ and $x'$ and observe that $m$ being odd implies that $q' - q \mod m \in \{ 0, \ldots, m-1 \}$ has a different parity than $q - q' \mod m$. Now consider the vertex $u \coloneqq (q, 0, \ldots, 0) \in V(\ZZ_m^n)$. Since it obviously belongs to $\iota_n(x)$, it must not belong to $\iota_n(x')$ by assumption, so $u_1 + \sum_{i = 1}^{n-1} \epsilon'_i u_{i+1} = q = q' + (q - q') \notin \{ q', q'+2, \ldots, q'+m-3 \}$ mod $m$. Subtracting $q'$, this excludes $0, 2, \ldots, m-3$ as possible values for the even number $q - q' \mod m \in \{ 0, \ldots, m-1 \}$, leaving only $q - q' \equiv m-1$ mod $m$ and thus, $q' \equiv q + 1 \mod m$.
	
	
Now if $\epsilon_1 \neq \epsilon'_1$, consider the vertex $v \coloneqq (q-1, \epsilon_1, 0, \ldots, 0) \in V(\ZZ_m^n)$. Then $v_1 + \sum_{i = 1}^{n-1} \epsilon_i v_{i+1} = q - 1 + \epsilon_1^2 = q \in \{ q, q+2, \ldots, q+m-3 \}$ mod $m$ and $v_1 + \sum_{i = 1}^{n-1} \epsilon'_i v_{i+1} = q - 1 - \epsilon_1^2 = q - 2 \in \{ q+1, q+3, \ldots, q+m-2 \} = \{ q', q'+2, \ldots, q'+m-3 \}$ mod $m$. This means that $v \in \iota_n(x) \cap \iota_n(x')$. By symmetry, the same conclusion also holds if $\epsilon_j \neq \epsilon'_j$ for some $j \in [n-1]$. Hence $\iota_n(x) \cap \iota_n(x') = \emptyset$ implies that $\epsilon_j = \epsilon'_j$ for all $j \in [n-1]$.
\end{proof}

\begin{lemma} \label{lem: no further maxindepsets}
For $m, n \in \NN$ with $m \geq 3$ odd, the map $\iota_n$ from \cref{lem: generate maxindepsets} is a bijection.
\end{lemma}

\begin{proof}
We again write $\ell \coloneqq \lfloor m/2 \rfloor$. It only remains to show that $\iota_n$ is surjective. Recall that while proving its well-definedness in \cref{lem: generate maxindepsets}, we already showed that $\abs I = m^{n-1} \ell$ for every $I \in \im(\iota_n) \subset \II^*(\ZZ_m^n)$, so the upper bound on the size of independent sets in \cref{lem: max size of maxindepset} is indeed attained for every $n$. We thus know that every $I \in \II^*(\ZZ_m^n)$ has $\abs I = m^{n-1} \ell$. 
	
We proceed by induction on $n$. For $n = 1$, the statement is trivial.
Suppose that for some $n > 1$, the function $\iota_{n-1}$ is a bijection, and let $I \in \II^*(\ZZ_m^n)$ be arbitrary. Partition $V(\ZZ_m^n)$ into $V_q \coloneqq \{ v \in V(\ZZ_m^n) \colon v_n = q \}$ for $q \in \ZZ_m$ and define the isomorphisms $\phi_q \colon \ZZ_m^n[V_q] \to \ZZ_m^{n-1} ,\, v \mapsto (v_1, \ldots, v_{n-1})$. We denote the independent set $\phi_q(I \cap V_q)$ as $I_q \in \II(\ZZ_m^{n-1})$ and observe that trivially, $\abs I = \sum_{q \in \ZZ_m} \abs {I_q}$. However, the sets $I_q$ can have at most $m^{n-2} \ell$ vertices by \cref{lem: max size of maxindepset}, so in order to achieve $\abs I = m^{n-1} \ell$, they must all be maximum independent sets in $\ZZ_m^{n-1}$. The induction hypothesis therefore guarantees that each $I_q$ has a preimage under $\iota_{n-1}$, which we will denote as $x_q \coloneqq \iota_{n-1}^{-1}(I_q) \in \ZZ_m \times \{ \pm 1 \}^{n-2}$. We refer to its first component as $p_q \in \ZZ_m$.

We now claim the following:

\begin{center}
\begin{minipage}{0.8\textwidth}
Except for the first component, all $x_q$ are identical. The first components satisfy $p_q = p_0 + q(p_1 - p_0)$ mod $m$ for all $q \in \ZZ_m$.
\end{minipage}
\end{center}

Let us first prove that $\{ p_{q-1}, p_{q+1} \} = \{ p_q-1, p_q+1 \}$ mod $m$ for every $q \in \ZZ_m$. It is easy to see that $\iota_{n-1}(x_q) = I_q$ and $\iota_{n-1}(x_{q - 1}) = I_{q - 1}$ must be disjoint as otherwise, $u \in I_q \cap I_{q - 1}$ for some $u \in V(\ZZ_m^{n-1})$ implies $(u, q), (u, q - 1) \in I$ and contradicts $I$ being independent. Similarly, $I_q \cap I_{q + 1} = \emptyset$. So \cref{lem: only disjoint if shifted} guarantees that all $x_q$ are identical except for their first component, for which $\{ p_{q-1}, p_{q+1} \} \subset \{ p_q - 1, p_q + 1 \}$ mod $m$ must hold.

For a proof by contradiction, assume that $p_{q-1} = p_{q+1} = p_q - 1$ mod $m$. Partition $V(\ZZ_m^n)$ into $V_w \coloneqq \{ v \in V(\ZZ_m^n) \colon (v_1, \ldots, v_{n-1}) = w \}$ for all $w \in V(\ZZ_m^{n-1})$. Observe that $\ZZ_m^n[V_w] \isom \ZZ_m^1$ for all $w \in V(\ZZ_m^{n-1})$. Consider now $\tilde w \coloneqq (p_q - 2, 0, \ldots, 0)$. We immediately observe that $p_q - 2 \notin \{ p_q, p_q + 2, \ldots, p_q + m - 3 \}$, so $\tilde w \notin \iota_{n-1}(x_q) = I_q = \phi_q(I \cap V_q)$ and $(\tilde w, q) \notin I$. However, we also see that $p_q - 2 \notin \{ p_q - 1, p_q + 1, \ldots, p_q + m - 4 \}$, so $\tilde w \notin \iota_{n-1}(x_{q \pm 1}) = I_{q \pm 1} = \phi_{q \pm 1}(I \cap V_{q \pm 1})$ and neither $(\tilde w, q-1)$ nor $(\tilde w, q+1)$ belong to $I$. But then $I \cap V_{\tilde w}$ must be an independent set in the path $\ZZ_m^n[V_{\tilde w}] \setminus \{ (\tilde w, q-1), (\tilde w, q), (\tilde w, q+1) \}$ on $m-3$ vertices and can thus have at most $\abs {I \cap V_{\tilde w}} \leq (m-3)/2 < \ell$ vertices. Summing $\abs {I \cap V_w}$ over all $w \in V(\ZZ_m^{n-1})$ then yields $\abs I = \sum_{w \in V(\ZZ_m^{n-1})} \abs {I \cap V_w} < m^{n-1} \ell$ in contradiction to $I \in \II^*(\ZZ_m^n)$. The same contradiction arises when assuming $p_{q-1} = p_{q+1} = p_q + 1$ mod $m$ and considering $\tilde w \coloneqq (p_q - 1, 0, \ldots, 0) \in V(\ZZ_m^{n-1})$ instead. This proves that $\{ p_{q-1}, p_{q+1} \} = \{ p_q-1, p_q+1 \}$ mod $m$ for every $q \in \ZZ_m$.

The equality $p_q = p_0 + q(p_1 - p_0)$ mod $m$ is trivially true for $q \in \{ 0, 1 \}$. In general, it follows by induction on $q$: Suppose that for some $q > 0$ the claim is true for $q-1$ and $q$. Then $\{ p_{q-1}, p_{q+1} \} = \{ p_q-1, p_q+1 \} = \{ p_0 + q(p_1 - p_0) - 1, p_0 + q(p_1 - p_0) + 1 \}$. Recalling that $p_1 - p_0 \in \{ \pm 1 \}$ by \cref{lem: only disjoint if shifted}, this set is exactly $\{ p_0 + (q-1)(p_1 - p_0), p_0 + (q+1)(p_1 - p_0) \}$ and as $p_{q-1} = p_0 + (q-1)(p_1 - p_0)$ by the induction hypothesis, we must have $p_{q+1} = p_0 + (q+1)(p_1 - p_0)$. This concludes the proof of the claim above.

\medskip
Finally, we show in the following that for $x_0 = (p_0, \epsilon_1, \ldots, \epsilon_{n-2})$ and $\epsilon_{n-1} \coloneqq p_0 - p_1 \in \{ \pm 1 \}$ mod $m$, we have $\iota_n(p_0, \epsilon_1, \ldots, \epsilon_{n-1}) = I$. We prove this equality for the respective intersections with $V_q = \{ v \in V(\ZZ_m^n) \colon v_n = q \}$ for all $q \in \ZZ_m$. So let $q \in \ZZ_m$ and $v \in V_q$ be arbitrary. Then by definition of $\iota_n$, we have $v \in \iota_n(p_0, \epsilon_1, \ldots, \epsilon_{n-1}) \cap V_q$ if and only if
\begin{align} \label{eq: iota(new x)}
	&v_1 + \sum_{i = 1}^{n-1} \epsilon_i v_{i+1} \notag\\
	={}&v_1 + \sum_{i = 1}^{n-2} \epsilon_i v_{i+1} + q(p_0 - p_1) \in \{ p_0, p_0 + 2, \ldots, p_0 + m - 3 \} \mod m
	\,.
\end{align}
On the other hand, we have $v \in I \cap V_q$ if and only if $(v_1, \ldots, v_{n-1}) \in \phi_q(I \cap V_q) = I_q = \iota_{n-1}(x_q)$ by definition of $\phi_q$. Also note that by the claim above, $x_0$ and $x_q$ share all but their first components, so $x_q = (p_q, \epsilon_1, \ldots, \epsilon_{n-1})$. This means that $v \in I \cap V_q$ is equivalent to
\[
	v_1 + \sum_{i = 1}^{n-2} \epsilon_i v_{i+1} \in \{ p_q, p_q + 2, \ldots, p_q + m - 3 \} \mod m
	\,.
\]
We can now use our claim to replace $p_q$ by $p_0 + q(p_1 - p_0) = p_0 - q(p_0 - p_1)$. Adding $q(p_0 - p_1)$ to both sides of this relation, we obtain exactly the condition in (\ref{eq: iota(new x)}). This shows that $\iota_n(x_0, \epsilon_{n-1}) \cap V_q = I \cap V_q$ for all $q \in \ZZ_m$ and so indeed, $\iota_n(x_0, \epsilon_{n-1}) = I$ holds. As $I \in \II^*(\ZZ_m^n)$ was chosen arbitrarily, this concludes the proof that $\iota_n$ is surjective and thus, a bijection.
\end{proof}

For $m = 3$, we can establish an upper bound on the size of the intersection of two maximum independent sets. Limiting this overlap is crucial to ensure that the process described above yields different independent sets $A \cup B$ when starting with different maximum independent sets $I$, which leads to the factor of $3 \cdot 2^{n-1}$ in the lower bound.
\begin{lemma} \label{lem: intersection maxindepsets}
For $n \in \NN$, let $I, I' \in \II^*(\ZZ_3^n)$ be distinct. Then $\abs {I \cap I'} \leq 3^{n-2}$.
\end{lemma}

\begin{proof}
Let $(q, \epsilon_1, \ldots, \epsilon_{n-1}) \coloneqq \iota_n^{-1}(I)$ and $(q', \epsilon'_1, \ldots, \epsilon'_{n-1}) \coloneqq \iota_n^{-1}(I')$. Suppose $v \in I \cap I'$, then adding and subtracting the conditions for $v \in I$ and $v \in I'$ yields
\[
	2v_1 + \sum_{i = 1}^{n-1} (\epsilon_i + \epsilon'_i) v_{i+1}
	= q + q' \mod 3
	\quad \text{and} \quad
	\sum_{i = 1}^{n-1} (\epsilon_i - \epsilon'_i) v_{i+1}
	= q - q' \mod 3
	\,.
\]
Let $J_+ \coloneqq \{ 0 \} \cup \{ i \in [n-1] \colon \epsilon_i = \epsilon'_i \}$ and $J_- \coloneqq \{ i \in [n-1] \colon \epsilon_i = -\epsilon'_i \}$. Then this is equivalent to $2 \sum_{i \in J_+} v_{i+1} = q + q'$ mod $3$ and $2 \sum_{i \in J_-} v_{i+1} = q - q'$ mod $3$. Multiplying both equations by $2 = -1$ mod $3$, we get
\[
	\sum_{i \in J_+} v_{i+1}
	= -q - q' \mod 3
	\quad \text{and} \quad
	\sum_{i \in J_-} v_{i+1}
	= -q + q' \mod 3
	\,.
\]
Observe that trivially, $J_+ \neq \emptyset$ and let $j_+ \coloneqq \max J_+$. Now suppose that that $J_- = \emptyset$. This immediately implies $\epsilon_i = \epsilon'_i$ for all $i$ and, by the second equation above, also that $q = q'$. Thus, $\iota_n^{-1}(I) = \iota_n^{-1}(I')$ in contradiction to $I \neq I'$. So $J_-$ is also nonempty and we can let $j_- \coloneqq \max J_-$. Then for each of the $3^{n-2}$ ways to choose the entries $v_{i+1}$ for $i \in \{ 0, 1, \ldots, n-1 \} \setminus \{ j_+, j_- \}$, there is at most one choice for $v_{j_++1}$ and $v_{j_-+1}$ (namely $v_{j_++1} = -q - q' - \sum_{i \in J_+ \setminus \{ j_+ \}}v_{i+1}$ mod $3$ and $v_{j_-+1} = -q + q' - \sum_{i \in J_- \setminus \{ j_- \}}v_{i+1}$ mod $3$) such that $v$ is in $I \cap I'$. This proves that $\abs {I \cap I'} \leq 3^{n-2}$.
\end{proof}

This concludes our analysis of maximum independent sets in $\ZZ_m^n$. We now turn our attention to the selection of defects from the subgraph $H_n$. Here, we observe that we can essentially assume the defects to be chosen independently as long as their number is small compared to the square root of the number of vertices available in $H_n$.
\begin{lemma} \label{lem: defect selection}
For every $n \in \NN$, let $M_n, K_n \in \NN$ and $H_n$ be a graph on $M_n$ vertices such that $(\Delta(H_n)+1) K_n^2 \leq o(M_n)$. Then for every $k \leq K_n$, we have
\[
	\abs {\II_k(H_n)}
	\geq (1 - o(1)) \frac {M_n^k}{k!}
	\,.
\]
\end{lemma}

\begin{proof}
Let $\epsilon > 0$ and choose $n$ large enough to guarantee $(\Delta(H_n) + 1)K_n^2 \leq \epsilon M_n$. Now select $k$ vertices from $H_n$ one after the other, excluding from the choices for the $(i+1)$-th vertex the $i$ vertices chosen in previous steps as well as the at most $\Delta(H_n)i$ vertices adjacent to any vertex chosen in previous steps. This ensures that the union of all vertices chosen is a set in $\II_k(H_n)$, but every such set is produced exactly $k!$ times. Having at least $M_n - (\Delta(H_n) + 1)i$ choices for the $(i+1)$-th vertex, we find that
\[
	\abs {\II_k(H_n)}
	\geq \frac 1{k!} \prod_{i = 0}^{k - 1} (M_n - (\Delta(H_n) + 1)i)
	= \frac {M_n^k}{k!} \prod_{i = 0}^{k - 1} \left( 1 - \frac {(\Delta(H_n) + 1)i}{M_n} \right)
	\,.
\]
Now we use the fact that $i \leq k \leq K_n$ and $(\Delta(H_n) + 1)K_n / M_n \leq \epsilon/K_n$ to obtain
\[
	\prod_{i = 0}^{k - 1} \left( 1 - \frac {(\Delta(H_n) + 1)i}{M_n} \right)
	\geq \prod_{i = 0}^{k - 1} \left( 1 - \frac {\epsilon}{K_n} \right)
	\geq \left( 1 - \frac {\epsilon}{K_n} \right)^{K_n}
	\geq 1 - \epsilon
\]
by the Bernoulli inequality.
\end{proof}

Finally, we will also use the following two immediate consequences of Chebyshev's inequality, applying it to a symmetric binomial and a Poisson distribution, respectively. Together, they allow us to bring the final lower bound into a closed form.
\begin{lemma} \label{lem: binomial}
For all $\epsilon > 0$, there exists an $n_0 \in \NN$ such that for all $n \geq n_0$ and all $y \in \NN$, the following holds:
\[
	2^{-y} \sum_{b = y/2 - n\sqrt{y/4}}^{y/2 + n\sqrt{y/4}} \binc{y}{b}
	\geq 1 - \epsilon
	\,.
\]
\end{lemma}

\begin{lemma} \label{lem: Poisson}
For all $\epsilon > 0$, there exists an $n_0 \in \NN$ such that for all $n \geq n_0$ and all $\lambda > 0$, the following holds:
\[
	\exp(-\lambda) \sum_{k = \lambda - n \sqrt {\lambda}}^{\lambda + n \sqrt {\lambda}} \frac {\lambda^k}{k!}
	\geq 1 - \epsilon
	\,.
\]
\end{lemma}

We can now combine all of this to prove the desired asymptotic lower bound of at least $(1-o(1)) \cdot 3 \cdot 2^{n-1} \cdot 2^{3^{n-1}} \cdot \exp( (3/2)^{n-1} )$ independent sets in $\ZZ_3^n$.

\begin{proof}[Proof of~\cref{thm: lowerboundK3}]
Let $x \coloneqq (q, \epsilon_1, \ldots, \epsilon_{n-1}) \in \ZZ_3 \times \{ \pm 1 \}^{n-1}$ be arbitrary and set $I \coloneqq \iota_n(x)$ as well as $U \coloneqq \{ v \in V(\ZZ_3^n) \colon v_1 + \sum_{i = 1}^{n-1} \epsilon_i v_{i+1} \in \{ q-2, q-1 \} \mod 3 \} = V(\ZZ_3^n) \setminus I$. Now consider the graph $H_n \coloneqq \ZZ_3^n[U]$ on $M_n \coloneqq 2 \cdot 3^{n-1}$ vertices. It is easy to see that $H_n$ is $n$-regular, so we have $\Delta(H_n) = n$. Let $\lambda \coloneqq (3/2)^{n-1}$ and $K_n \coloneqq \lambda + n \sqrt{\lambda}$. Then we immediately observe that $(\Delta(H_n) + 1)K_n^2 \leq O(n\lambda^2) = O(n(9/4)^n) = o(3^n) = o(M_n)$. Selecting any $k \in [\lambda - n \sqrt{\lambda}, K_n]$, we thus have $\abs {\II_k(H_n)} \geq (1 - o(1)) \frac {M_n^k}{k!}$ by \cref{lem: defect selection}.

Now consider some set $A \in \II_k(H_n)$. As every vertex in $A \subset U$ has exactly $n$ edges to $I$, we have $\abs {N_{\ZZ_3^n}(A, I)} \leq n \abs A = nk$. Let $y \coloneqq 3^{n-1} - nk$ and select some $B \in I \setminus N_{\ZZ_3^n}(A, I)$ with $\abs B \in [y/2 - n \sqrt{y/4}, y/2 + n \sqrt{y/4}]$. Then $A \cup B \in \II(\ZZ_3^n)$ by construction and the number of distinct such $A \cup B$ is at least
\begin{align*}
	\sum_{k = \lambda - n\sqrt{\lambda}}^{\lambda + n\sqrt{\lambda}} \abs {\II_k(H_n)} \sum_{b = y/2 - n\sqrt{y/4}}^{y/2 + n\sqrt{y/4}} \binc{y}{b}
	&\geq (1 - o(1)) \sum_{k = \lambda - n\sqrt{\lambda}}^{\lambda + n\sqrt{\lambda}} \frac {M_n^k \cdot 2^{y}}{k!} \\
	&= (1 - o(1)) \cdot 2^{3^{n-1}} \sum_{k = \lambda - n\sqrt{\lambda}}^{\lambda + n\sqrt{\lambda}} \frac {(2 \cdot 3^{n-1})^k \cdot 2^{-nk}}{k!} \\
	&= (1 - o(1)) \cdot 2^{3^{n-1}} \sum_{k = \lambda - n\sqrt{\lambda}}^{\lambda + n\sqrt{\lambda}} \frac {\lambda^k}{k!} \\
	&\geq (1 - o(1)) \cdot 2^{3^{n-1}} \exp \left( \left( \tfrac 32 \right)^{n-1} \right)
\end{align*}
by \cref{lem: binomial,lem: Poisson}.

As $I \in \II^*(\ZZ_3^n)$ was arbitrary, we can actually obtain the $\abs {\II^*(\ZZ_3^n)}$-fold of this bound, which is exactly the desired statement, if we can show that an independent set $A \cup B$ produced by the above process starting from $I \in \II^*(\ZZ_3^n)$ cannot also be written as $A' \cup B'$ produced starting from $I' \in \II^*(\ZZ_3^n) \setminus \{ I \}$. For a proof by contradiction, suppose this were false and $A \cup B = A' \cup B'$. Then we would have
\[
	\abs {I \cap I'} 
	\geq \abs {B \cap B'} 
	\geq \abs B - \abs {A'}
	\geq y/2 - n\sqrt{y/4} - K_n 
	\geq 3^{n-1}/2 - o(3^n)
\]
and could ensure $\abs {I \cap I'} \geq 3^{n-1}/3 = 3^{n-2}$ by choosing $n$ sufficiently large. This, however, contradicts \cref{lem: intersection maxindepsets} and thus finishes the proof.
\end{proof}

A first natural guess would be to assume that the cases $m \geq 5$ can be treated similarly as for $m = 3$.
However, for most independent sets, the number of defects becomes so large that when choosing them as an independent subset of the graph~$H_n$ induced by two adjacent partition classes, we cannot essentially ignore the edges between these two classes anymore. In the terminology of \cref{lem: defect selection}, this graph $H_n$ has $M_n = 2m^{n-1}$ vertices, while we need to allow defect sets of size up to $K_n \geq (m/2)^{n-1}$. This violates the assumption $(\Delta(H_n) + 1)K_n^2 \leq o(M_n)$ of \cref{lem: defect selection} for $m \geq 5$.

For general odd $m$, we therefore only derive a less precise lower bound. 
Its proof relies on the following lemma, which functions as a rough lower bound on the number of independent sets $A$ in the graph $H = H_n$.

\begin{lemma} \label{lem: estimate partition function}
Let $H$ be a graph and $\lambda \in (0, 1/2]$. Then for $p \coloneqq \lambda/(1 + \lambda)$, we have
\[
	\sum_{A \in \II(H)} \lambda^{\abs A}
	\geq \frac {(1 - p^2)^{\abs {E(H)}}} {(1-p)^{\abs {V(H)}}}
	\geq \exp \left( p \abs {V(H)} - 2p^2 \abs {E(H)} \right)
	\,.
\]
\end{lemma}

\begin{proof}
Consider the random subset $X \subset V(H)$ that arises from selecting every vertex of $H$ independently with probability $p$. We calculate the probability of $X$ being independent in two different ways. On the one hand, we have
\begin{align*}
	\PP[X \in \II(H)]
	= \sum_{A \in \II(H)} \PP[X = A]
	&= \sum_{A \in \II(H)} p^{\abs A} (1-p)^{\abs {V(H)} - \abs A} \\
	&= (1-p)^{\abs {V(H)}} \sum_{A \in \II(H)} \left( \frac p{1-p} \right)^{\abs A}
	\,.
\end{align*}
Observe that $p/(1-p) = \lambda$. On the other hand, we can also consider all the edges $uv \in E(H)$ and calculate the probability that for none of them, both endpoints belong to $X$. This yields
\begin{align*}
	\PP[X \in \II(H)]
	= \PP \left[ \bigcap_{uv \in E(H)} \{ u \in X \wedge v \in X \}^c \right]
	&\geq \prod_{uv \in E(H)} (1 - \PP[u \in X \wedge v \in X]) \\
	&= (1 - p^2)^{\abs {E(H)}}
	\,,
\end{align*}
where we have fixed an enumeration of $V(H) = \{ v_1, \ldots, v_k \}$ and identified $X \subset V(H)$ with $x \in \{ 0, 1 \}^k$ defined by $x_i = 1$ if $v_i \in X$, so we can repeatedly apply \cref{lem: FKG}. 
Note that the partial order $\leq$ on $\{ 0, 1 \}^k$ is just the subset relation and for all $E \subset E(H)$, the event $\bigcap_{uv \in E} \{ u \in X \wedge v \in X \}^c$ is decreasing.
Combining both observations proves the first statement claimed. 

In order to obtain the second inequality, we note that $p, p^2 \in (0, 1/2)$ since $0 < p < \lambda$. This allows us to use the geometric series to calculate
	\[
	\frac 1{1 - p}
	= \sum_{i = 0}^\infty p^i
	\geq \sum_{i = 0}^\infty \frac {p^i}{i!}
	= e^p
	\]
and also bound $\ln \left( (1 - p^2)^{\abs {E(H)}} \right) = \abs {E(H)} \ln (1 - p^2)$ from below using
	\[
	\ln (1 - p^2)
	= \sum_{k = 1}^\infty (-1)^{k+1} \frac {(-p^2)^k}k
	= -\sum_{k = 1}^\infty \frac {p^{2k}}k
	\geq -\sum_{k = 1}^\infty p^{2k}
	= - \frac {p^2}{1 - p^2}
	\geq -2p^2
	\,.
\]
This finishes the proof.
\end{proof}

We are now ready to prove \cref{thm: lowerboundKm}. Recall that we need to show that for $n, m \in \NN$ with $m$ odd and $\ell \coloneqq \lfloor m/2 \rfloor$, we have
\[
	\abs {\II(\ZZ_m^n)}
	\geq 2^{\ell m^{n-1}} \exp \left( (1 - o(1)) \left( \frac m2 \right)^{n-1} \right)
	\,.
\]

\begin{proof}[Proof of \cref{thm: lowerboundKm}]
Fix a maximum independent set $I \in \II^*(\ZZ_m^n)$ and let $H \subset \ZZ_m^n$ be the subgraph induced by the two adjacent partition classes in $\overline I$. We can obviously obtain a set of pairwise distinct independent sets in $\ZZ_m^n$ by considering all combinations $A \cup B$ of $A \in \II(H)$ and $B \subset I \setminus N_{\ZZ_m^n}(A, I)$.
As $\abs {N_{\ZZ_m^n}(A, I)} \leq \sum_{a \in A} \abs {N_{\ZZ_m^n}(a, I)} \abs A = n \abs A$, we obtain
\[
	\abs {\II(\ZZ_m^n)}
	\geq \sum_{A \in \II(H)} 2^{\abs {I \setminus N_{\ZZ_m^n}(A, I)}}
	\geq \sum_{A \in \II(H)} 2^{\ell m^{n-1} - n \abs A}
	= 2^{\ell m^{n-1}} \sum_{A \in \II(H)} (2^{-n})^{\abs A}
	\,.
\]
We now apply \cref{lem: estimate partition function} with $\lambda \coloneqq 2^{-n}$ and $p \coloneqq 2^{-n}/(1 + 2^{-n})$, which satisfies $(1 - o(1)) 2^{-n} \leq p \leq 2^{-n}$. We observe that $\abs {V(H)} = 2m^{n-1}$ and $\abs {E(H)} = nm^{n-1}$, so we obtain
\[
	\sum_{A \in \II(H)} (2^{-n})^{\abs A}
	\geq \exp \left( p \abs {V(H)} - 2p^2 \abs {E(H)} \right)
	= \exp \left( 2pm^{n-1} (1 - pn) \right)
	\,.
\]
The statement follows from $2pm^{n-1} \geq (1 - o(1))(m/2)^{n-1}$ and $pn \leq n2^{-n} \leq o(1)$.
\end{proof}

Note that in contrast to \cref{thm: lowerboundK3}, the lower bound in \cref{thm: lowerboundKm} does not contain a factor of $m \cdot 2^{n-1}$ representing the choice of $I \in \II(\ZZ_m^n)$ anymore. This is due to the fact that $m \cdot 2^{n-1} = o(\exp ((m/2)^n))$, so the lost factor is anyway smaller than the error allowed in \cref{thm: lowerboundKm}.

\section{Cluster Expansion}\label{sec:clusterexpansion}
One of the most powerful tools for counting independent sets is the cluster expansion from statistical physics. 
In particular, it is often able to yield much more detailed asymptotic formulas for the number of independent sets in a graph.
Since the method is quite well-established, we omit the technical details and instead focus on how to employ this method in the current setting.
We first introduce the polymer model $\PM(\ZZ_m^n, I)$ and establish its connection to $\abs {\II(\ZZ_m^n)}$, before we direct our attention to the clusters and actually calculate some initial terms of the cluster expansion.

\subsection{The polymer model}
Let $G$ be a graph and fix an independent set $I \in \II(G)$.
We say that a set $S \subset \overline I \coloneqq V(G) \setminus I$ is \emph{$(2, I)$-linked} if $G[S \cup N_G(S, I)]$ is connected.
We define the following \emph{polymer model} $\PM(G, I)$: 
A \emph{polymer} is a $(2, I)$-linked subset of $\overline I$ that is independent in $G$. 
Two polymers $S, T$ are \emph{compatible} if $S \cup T$ is independent in $G$, but not $(2, I)$-linked. 
We also write this as $S \sim T$.
The \emph{weight} of a polymer $S$ is defined as $w(S) \coloneqq 2^{-\abs {N_G(S, I)}}$.

For every polymer model $\PM$, one can consider the associated \emph{partition function} $Z_{\PM}
\coloneqq \sum_{\mathcal S} \prod_{S \in \mathcal S} w(S)$, where the sum is over all sets $\mathcal S$ of pairwise compatible polymers.
In the case of the polymer model $\PM = \PM(G, I)$ defined above, this partition function is actually counting independent sets in $G$.

\begin{lemma} \label{lem: indepsets polymer model}
Let $G$ be a graph and $I \in \II(G)$. Then $\abs {\II(G)} = 2^{\abs I} Z_{\PM(G, I)}$.
\end{lemma}

\begin{proof}
For every independent set in $G$, there is exactly one way to write it as the union $A \cup B$ of $A \in \II(G[\overline I])$ and $B \subset I \setminus N_G(A, I)$.
This yields
\[
	\abs {\II(G)}
	= \sum_{A \in \II(G[\overline I])} 2^{\abs I - \abs {N_G(A, I)}}
	= 2^{\abs I} \sum_{A \in \II(G[\overline I])} 2^{- \abs {N_G(A, I)}}
	\,.
\]

Our next goal is to show that there is a one-to-one correspondence between independent sets $A \in \II(G[\overline I])$ and sets $\mathcal S$ of pairwise compatible polymers, which enables us to sum over all such sets $\mathcal S$ instead. 
If $G$ is bipartite and $I, \overline I$ are its partition classes (as is usually the case in the literature on $\ZZ_m^n$ with even $m$), this is straightforward. 
Meanwhile, our setting allows edges inside of $\overline I$ and thus requires the more delicate polymer definition above.
Therefore, a formal proof of the desired correspondence seems warranted to us.

To this end, consider the following two functions that map independent sets $A \in \II(G[\overline I])$ to sets $\mathcal S$ of pairwise compatible polymers and vice-versa: 
Given an independent set $A \in \II(G[\overline I])$, decompose the graph $G[A \cup N_G(A, I)]$ into its $\ell \geq 0$ connected components $C_1, \ldots, C_\ell$ and define $\Phi(A) \coloneqq \{ V(C_i) \cap \overline I \mid i \in [\ell] \}$.
Given a set $\mathcal S$ of pairwise compatible polymers, define $\Psi(\mathcal S) \coloneqq \bigcup_{S \in \mathcal S} S$.

\begin{claim}\label{claim: indepsets<->polymers}
The functions $\Phi$ and $\Psi$ are well-defined and the inverse of each other.
\end{claim}

\begin{proofclaim}
Let $A \in \II(G[\overline I])$ be arbitrary and $C_1, \ldots, C_\ell$ be the connected components of $G[A \cup N_G(A, I)]$. Since $(A \cup N_G(A, I)) \cap \overline I = A$, we find that the $S_i \coloneqq V(C_i) \cap \overline I$ partition $A$. On the one hand, this immediately proves that $\Psi (\Phi(A)) = A$. On the other hand, this also guarantees that the $S_i$ inherit from $A$ that they are subsets of $\overline I$ that are independent in $G$. By construction, each $G[S_i \cup N_G(S_i, I)] = C_i$ is connected, so the $S_i$ are indeed $(2, I)$-linked and thus polymers. 

In order to see that the $S_i$ are also pairwise compatible, note that $S_i \cup S_j \subset A$ must be independent in $G$ and thus assume that it is still $(2, I)$-linked. This means that $G[S_i \cup N_G(S_i, I) \cup S_j \cup N_G(S_j)]$ is connected and thus belongs to the same connected component $C_i = C_j$ of $G[A \cup N_G(A, I)]$. As desired, $i = j$ follows. This shows that $\Phi$ is indeed well-defined.

For the inverse direction, let $\mathcal S$ be a set of pairwise compatible polymers. Then each $S \in \mathcal S$ is a subset of $\overline I$ that is independent in $G$. Their union must therefore also be a subset of $\overline I$. If it were not independent, there would be $u, v$ in distinct $S, T \in \mathcal S$ with $uv \in E(G)$. However, then $S \cup T$ would not be independent and $S, T$ would therefore not be compatible, a contradiction. This proves that $\Psi$ is well-defined.
	
In order to see that $\Phi(\Psi(\mathcal S)) = \mathcal S$, first note that by compatibility, both $\Phi(\Psi(\mathcal S))$ and $\mathcal S$ partition $A \coloneqq \Psi(\Phi(\Psi(\mathcal S))) = \Psi(\mathcal S)$. It therefore suffices to show that every $S \in \mathcal S$ is a subset of some $S' \in \Phi(A)$ and vice-versa: 
	\begin{itemize}
		\item Let $S \in \mathcal S$ be arbitrary. By its $(2, I)$-linkedness, $G[S \cup N_G(S, I)]$ is a connected subgraph of $G[A \cup N_G(A, I)]$ and must therefore belong to a single connected component of $G[A \cup N_G(A, I)]$. This shows that the intersection $(S \cup N_G(S, I)) \cap \overline I = S$ is contained in some element of $\Phi(A)$. 
		\item On the other hand, consider some connected component $C$ of $G[A \cup N_G(A, I)]$ and suppose $C \cap \overline I$ intersects with multiple polymers in $\mathcal S$. Let $S, T \in \mathcal S$ be two such polymers, which by connectedness of $C$, can be chosen such that $S \cup N_G(S, I)$ and $T \cup N_G(T, I)$ are adjacent in $G$. As $I$ is independent, any connecting edge is either between $S$ and $T$ or establishes an intersection of $N_G(S, I)$ and $N_G(T, I)$. Either way, it contradicts $S$ and $T$ being compatible.
	\end{itemize}
This shows that $\Phi(\Psi(\mathcal S)) = \mathcal S$ and thus finishes the proof of the claim.
\end{proofclaim}

The $(2, I)$-linkedness of every $S \in \Phi(A)$ now guarantees that the $N_G(S, I)$ partition $N_G(A, I)$. This allows us to write
\[
	2^{- \abs {N_G(A, I)}}
	= 2^{- \sum_{S \in \Phi(A)} \abs {N_G(S, I)}}
	= \prod_{S \in \Phi(A)} 2^{- \abs {N_G(S, I)}}
	= \prod_{S \in \Phi(A)} w(S)
	\,.
\]
Having established in \cref{claim: indepsets<->polymers} that $\Phi$ bijectively maps independent sets $A \in \II(G[\overline I])$ to sets $\mathcal S$ of pairwise compatible polymers, we conclude that
\[
	\sum_{A \in \II(G[\overline I])} 2^{- \abs {N_G(A, I)}}
	= \sum_{A \in \II(G[\overline I])} \prod_{S \in \Phi(A)} w(S)
	= \sum_{\mathcal S} \prod_{S \in \mathcal S} w(S)
	= Z_{\PM(G, I)}
	\,,
\]
which finishes the proof.
\end{proof}

\subsection{Cluster expansion}
Let $\PM$ be a polymer model with compatibility relation~$\sim$. A \emph{cluster} of $\PM$ is a vector $\Gamma = (S_1, S_2, \ldots, S_k)$ of $k \geq 1$ (not necessarily distinct) polymers of $\PM$ such that the corresponding \emph{incompatibility graph} $H_\Gamma = ([k], \{ ij \colon S_i \not \sim S_j \})$ is connected. The \emph{size} of a cluster $\Gamma = (S_1, S_2, \ldots, S_k)$ is defined as $\Abs \Gamma \coloneqq \sum_{i = 1}^k \abs {S_i}$. Let $\CM(\PM)$ denote the (infinite) set of all clusters and write $\CM_r(\PM) \coloneqq \{ \Gamma \in \CM(\PM) \colon \Abs \Gamma = r \}$. We now define the \emph{Ursell function} of a graph $H$ as
\[
	\phi(H)
	\coloneqq \frac 1{\abs {V(H)}!} \sum_{\substack{\text{spanning, connected}\\\text{subgraphs $F \subset H$}}} (-1)^{\abs {E(F)}}
	\,.
\]
With this, we can write the logarithm of the partition function $Z_{\PM}$ as the following formal power series, which is also known as the \emph{cluster expansion} of $\PM$:
\[
	\log Z_{\PM}
	= \sum_{r = 1}^\infty L_r(\PM)
	\qquad \text{with} \qquad
	L_r(\PM) \coloneqq \sum_{\substack{\Gamma \in \CM_r(\PM) \\ \Gamma = (S_1, \ldots, S_k)}} \phi(H_\Gamma) \prod_{i = 1}^k w(S_i)
	\,.
\]
We apply this to $\PM = \PM(\ZZ_m^n, I)$, choosing an arbitrary maximum independent set $I \in \II^*(\ZZ_m^n)$. This requires us to figure out the different types of clusters that exist in this model as well as calculate their contribution to the cluster expansion. For small $r$, this is straightforward enough to do.

\begin{theorem}\label{thm: L1L2}
For $n, m \in \NN$ with $m \geq 5$ odd and $I \in \II^*(\ZZ_m^n)$, we have
\begin{align*}
	L_1(\PM(\ZZ_m^n, I)) = m^{n-1} \bigg( & 2 \cdot 2^{-n} + \frac {m-3}2 \cdot 2^{-2n} \bigg) \\
	L_2(\PM(\ZZ_m^n, I)) = m^{n-1} \bigg( &\big( n^2 - 2n - 1 \big) \cdot 2^{-2n} + \big( 3n^2 - n \big) \cdot 2^{-3n} \\
	&{} + \bigg( \frac {3m - 12}2 n^2 + \frac{7 - 2m}2 n - \frac {m-3}4 \bigg) \cdot 2^{-4n} \bigg)
	\,.
\end{align*}
\end{theorem}

\begin{proof}
By symmetry, it suffices to consider $I = \iota_n(1, 1, \ldots, 1)$. Let $\PM \coloneqq \PM(\ZZ_m^n, I)$ be the corresponding polymer model and write $\EE(p) \coloneqq \{ v \in V(\ZZ_m^n) \colon \sum_{i = 1}^n v_i = p \mod m \}$ for $p \in \ZZ_m$. This means that $I = \EE(1) \cup \EE(3) \cup \ldots \cup \EE(m-2)$, while $\overline I = \EE(0) \cup \EE(2) \cup \ldots \cup \EE(m-1)$. Note that vertices in $X \coloneqq \EE(0) \cup \EE(m-1)$ have $n$ neighbours in $I$, whereas vertices in $Y \coloneqq \overline I \setminus X$ have $2n$ neighbours in $I$. 

For size $1$, each cluster $\Gamma = (\{ v \}) \in \CM_1(\PM)$ consists of a one single-vertex polymer and has $\phi(H_\Gamma) = 1$. Consequently, there are $2 \cdot m^{n-1}$ clusters with $v \in X$ and weight $1 \cdot 2^{-n}$ as well as $\frac {m-3}2 \cdot m^{n-1}$ clusters with $v \in Y$ and weight $1 \cdot 2^{-2n}$.

For size $2$, there are nine different types of clusters $\Gamma = (\{ v \}, \{ w \}) \in \CM_2(\PM)$ that consist of two single-vertex polymers. Their incompatibility graph $H_\Gamma$ is a single edge, so $\phi(H_\Gamma) = -\frac 12$ for all of them.
\begin{enumerate}
	\item $v \in X$ and $w = v$: $2 \cdot m^{n-1}$ clusters of weight $-\frac 12 \cdot 2^{-2n}$.
	\item $v \in Y$ and $w = v$: $\frac {m-3}2 \cdot m^{n-1}$ clusters of weight $-\frac 12 \cdot 2^{-4n}$.
	\item $v \in \EE(0)$ and $w = v - e_i$ or $v \in \EE(m-1)$ and $w = v + e_i$ for $i \in [n]$: $2 \cdot m^{n-1} \cdot n$ clusters of weight $-\frac 12 \cdot 2^{-2n}$.
	\item $v \in X$ and $w = v + e_i - e_j$ for $i \neq j \in [n]$: $2 \cdot m^{n-1} \cdot n(n-1)$ clusters of weight $-\frac 12 \cdot 2^{-2n}$.
	\item $v \in Y$ and $w = v + e_i - e_j$ for $i \neq j \in [n]$: $\frac {m-3}2 \cdot m^{n-1} \cdot n(n-1)$ clusters of weight $-\frac 12 \cdot 2^{-4n}$.
	\item $v \in \EE(0) \cup \EE(m-3)$ and $w = v + 2e_i$ or $v \in \EE(2) \cup \EE(m-1)$ and $w = v - 2e_i$ for $i \in [n]$: $4 \cdot m^{n-1} \cdot n$ clusters of weight $-\frac 12 \cdot 2^{-3n}$.
	\item $v \in \EE(2) \cup \ldots \cup \EE(m-5)$ and $w = v + 2e_i$ or $v \in \EE(4) \cup \ldots \cup \EE(m-3)$ and $w = v - 2e_i$ for $i \in [n]$: $(m - 5) \cdot m^{n-1} \cdot n$ clusters of weight $-\frac 12 \cdot 2^{-4n}$.
	\item $v \in \EE(0) \cup \EE(m-3)$ and $w = v + e_i + e_j$ or $v \in \EE(2) \cup \EE(m-1)$ and $w = v - e_i - e_j$ for $i \neq j \in [n]$: $4 \cdot m^{n-1} \cdot \frac {n(n-1)}2$ clusters of weight $-\frac 12 \cdot 2^{-3n}$.
	\item $v \in \EE(2) \cup \ldots \cup \EE(m-5)$ and $w = v + e_i + e_j$ or $v \in \EE(4) \cup \ldots \cup \EE(m-3)$ and $w = v - e_i - e_j$ for $i \neq j \in [n]$: $(m - 5) \cdot m^{n-1} \cdot \frac {n(n-1)}2$ clusters of weight $-\frac 12 \cdot 2^{-4n}$.
\end{enumerate}

The remaining clusters $\Gamma = (\{ v, w \}) \in \CM_2(\PM)$ consist of a single two-vertex polymer and thus satisfy $\phi(H_\Gamma) = 1$. 
In fact, all such polymers $\{ v, w \}$ correspond to two clusters $(\{ v \}, \{ w \})$ and $(\{ w \}, \{ v \})$ of the same type in the list above. 
It is easy to see that for clusters $(\{ v \}, \{ w \})$ of type (1), (2), and (3), the union of their vertices is not a two-vertex polymer in $\PM$. 
For all the other types, however, $(\{ v, w \})$ is indeed a valid cluster in $\CM_2(\PM)$. 
We thus obtain the number of these clusters by dividing the number above by $2$.
The weights can be calculated by replacing $-\frac 12$ by $1$ and adding a factor of $2$ for every shared neighbour of $v$ and $w$ in $I$.
\begin{enumerate}[start = 4]
	\item $v$ and $w$ share $v + e_i$ if $v \in \EE(0)$ or $v - e_j$ if $v \in \EE(m-1)$: $m^{n-1} \cdot n(n-1)$ clusters of weight $2 \cdot 2^{-2n}$.
	\item $v$ and $w$ share $v + e_i$ and $v - e_j$: $\frac {m-3}4 \cdot m^{n-1} \cdot n(n-1)$ clusters of weight $4 \cdot 2^{-4n}$.
	\item $v$ and $w$ share $v + e_i$ if $v \in \EE(0) \cup \EE(m-3)$ or $v - e_i$ if $v \in \EE(2) \cup \EE(m-1)$: $2 \cdot m^{n-1} \cdot n$ clusters of weight $2 \cdot 2^{-3n}$.
	\item $v$ and $w$ share $v + e_i$ if $v \in \EE(2) \cup \ldots \cup \EE(m-5)$ or $v - e_i$ if $v \in \EE(4) \cup \ldots \cup \EE(m-3)$: $\frac {m-5}2 \cdot m^{n-1} \cdot n$ clusters of weight $2 \cdot 2^{-4n}$.
	\item $v$ and $w$ share $v + e_i$ and $v + e_j$ if $v \in \EE(0) \cup \EE(m-3)$ or $v - e_i$ and $v - e_j$ if $v \in \EE(2) \cup \EE(m-1)$: $2 \cdot m^{n-1} \cdot \frac {n(n-1)}2$ clusters of weight $4 \cdot 2^{-3n}$.
	\item $v$ and $w$ share $v + e_i$ and $v + e_j$ if $v \in \EE(2) \cup \ldots \cup \EE(m-5)$ or $v - e_i$ and $v - e_j$ if $v \in \EE(4) \cup \ldots \cup \EE(m-3)$: $\frac {m-5}2 \cdot m^{n-1} \cdot \frac {n(n-1)}2$ clusters of weight $4\cdot 2^{-4n}$.
\end{enumerate}

Grouping by exponent of the $2^{-n}$-term and simplifying yields both formulas.
\end{proof}

Note that every vertex in $\overline I$ has either $n$ or $2n$ edges to $I$, so every polymer $S$ has weight $w(S) \leq 2^{- \abs S n}$ and every cluster $\Gamma \in \CM_r(\PM)$ contributes at most $O(2^{-rn})$ to $\log Z_\PM$. 
For small $r$, the number of clusters in $\CM_r(\PM)$ is obviously bounded by $O(p_r(n) \cdot m^n)$ for some polynomial $p_r$, so the terms $L_r(\PM(\ZZ_m^n, I))$ with $m/2^r < 1$ become negligible as $n \to \infty$. 
In order to establish convergence of the cluster expansion, however, one has to make the same argument for arbitrarily large $r$.
Since we are unable to achieve this, the approach above only yields a conjecture on which clusters are relevant, but falls short of a proof.

In the case of $m \in \{ 5, 7 \}$, we have $m/2^3 < 1$, so only clusters of size at most $2$ should be relevant. 
It is also sensible to assume that starting with distinct maximum independent sets $I, I' \in \II^*(\ZZ_m^n)$ will again result in $A \cup B \neq A' \cup B'$ for all but a negligible fraction of combinations of $A \in \II(G[\overline I])$ and $B \in I \setminus N_G(A, I)$ as well as $A' \in \II(G[\overline {I'}])$ and $B' \in I' \setminus N_G(A', I')$.
Therefore, the calculation in \cref{thm: L1L2} together with \cref{lem: no further maxindepsets,lem: indepsets polymer model} naturally leads to the following conjecture.

\begin{conjecture}\label{conj: pfct57}
\begin{align*}
	\abs {\II(\ZZ_5^n)}
	&= (1 \pm o(1)) \cdot 5 \cdot 2^{n-1} \cdot 2^{2 \cdot 5^{n-1}} \cdot \exp \left( \left( \frac 52 \right)^{n-1} + \frac {n^2 - 2n}4 \left( \frac 54 \right)^{n-1} \right) \\
	\abs {\II(\ZZ_7^n)}
	&= (1 \pm o(1)) \cdot 7 \cdot 2^{n-1} \cdot 2^{3 \cdot 7^{n-1}} \cdot \exp \left( \left( \frac 72 \right)^{n-1} + \frac {n^2 - 2n + 1}4 \left( \frac 74 \right)^{n-1} \right)
	\,.
\end{align*}
\end{conjecture}

\section{Isoperimetric inequalities}\label{sec:isoperimetry}
The reasoning behind the proof of the upper bound in \cite{Gal19} is actually quite similar to the initial argument of \cref{lem: indepsets polymer model}: 
Having fixed a (maximum) independent set $I \in \II^*(\ZZ_m^n)$, every independent set in $\ZZ_m^n$ can be partitioned into its intersections $A, B$ with $\overline I$ and $I$, respectively.
Since both inherit independence in $\ZZ_m^n$, the former is a set $A \in \II(\ZZ_m^n[\overline I])$. 
By independence of $I$, any $B \subset I$ is automatically independent, but once $A$ is known, $B$ must also satisfy $B \subset I \setminus N_{\ZZ_m^n}(A, I)$, leaving exactly $2^{\abs I - \abs {N_{\ZZ_m^n}(A, I)}}$ choices.
For $m \geq 3$ odd and $\ell \coloneqq \lfloor m/2 \rfloor$, we have already established in \cref{sec:lower-bounds} that all $I \in \II^*(\ZZ_m^n)$ have the same structure, so it suffices to only look at one representative. 
With the notation of \cref{lem: generate maxindepsets}, we choose $I_0 \coloneqq \iota_n(0, 1, \ldots, 1)$.
Also using $\abs {\II^*(\ZZ_m^n)} = m \cdot 2^{n-1}$ and $\abs I = \ell m^{n-1}$ from \cref{sec:lower-bounds}, we therefore immediately observe that
\[
	\abs {\II(\ZZ_m^n)}
	\leq \sum_{I \in \II^*(\ZZ_m^n)} \sum_{A \in \II(\ZZ_m^n[\overline I])} 2^{\abs I - \abs {N_{\ZZ_m^n}(A, I)}}
	\leq m \cdot 2^{n-1} \cdot 2^{\ell m^{n-1}} \cdot \sum_{A \in \II(\ZZ_m^n[\overline {I_0}])} 2^{-\abs {N_{\ZZ_m^n}(A, I_0)}}
	\,.
\]

It remains to bound the sum on the right from above. 
For this, we require a lower bound on the neighbourhood size of certain independent sets $A$. 
At its core, this is asking for an isoperimetric inequality in the graph $\ZZ_m^n$, that is some lower bound on $\abs {N_{\ZZ_m^n}(A)}$ in terms of $\abs A$. 
Fixing all but one coordinate, $\abs {N_{\ZZ_m^n}(A)} \geq \abs A$ is trivial to obtain. 
In order to make progress towards an upper bound, however, guaranteeing that this neighbourhood is actually slightly larger than $A$ (for $A$ not too large) seems necessary.
It is worth noting that proving such an inequality is also the crucial step in establishing convergence of the cluster expansion, since we again need to limit how many independent sets $A$ can contribute a weight of $w(A) = 2^{-\abs {N_{\ZZ_m^n}(A, I)}}$ to $\log Z_{\PM(\ZZ_m^n, I)}$.

In the following, we will prove an isoperimetric inequality that constitutes first progress towards an upper bound for $\II(\ZZ_m^n)$ with $m \geq 3$ odd, but unfortunately is not strong enough yet. Our proof adapts the approach of \cite[Lemma~6.1]{JK20} to the case of odd sidelength. Since in the construction of the abovementioned maximum independent set $I_0 = \iota_n(0, 1, \ldots, 1) \in \II^*(\ZZ_m^n)$, vertices $v \in V(\ZZ_m^n)$ are classified according to the value of $\sum_{i = 1}^n v_i$ mod $m$, it is helpful to refer to this number as the \emph{class} $\cls(v)$ of $v$. This partitions $V(\ZZ_m^n)$ into $\EE_m^n(p) \coloneqq \{ v \in V(\ZZ_m^n) \colon \cls(v) = p \}$ for $p \in \ZZ_m$. Note that with this notation, we have $I_0 = \EE_m^n(0) \cup \EE_m^n(2) \cup \ldots \cup \EE_m^n(m-3)$.

For the sake of simplicity, we write $N_n$ instead of $N_{\ZZ_m^n}$. For arbitrary subsets $A \subset V(\ZZ_m^n)$ and $q \in \ZZ_m$, we define $A_q \coloneqq \{ v \in V(\ZZ_m^{n-1}) \colon (v, q) \in A \}$ and hence obviously $\abs A = \sum_{q \in \ZZ_m} \abs {A_q}$. The following lemma is the central step in our induction.

\begin{lemma}\label{lem: slice-est}
For $m, n \in \NN$ with $m \geq 3$ odd, consider $p, q \in \ZZ_m$ and $A \subset \EE_m^n(p)$. Then
\[
	\abs {N_n(A, \EE_m^n(p \pm 1))_q} 
	\geq \max \left\lbrace \abs {N_{n-1}(A_q, \EE_m^{n-1}(p - q \pm 1))}, \abs {A_{q \mp 1}} \right\rbrace
	\,.
\]
\end{lemma}

\begin{proof}
By symmetry, it suffices to prove the statement for $N_n(A, \EE_m^n(p + 1))_q$. We show that in fact, this set contains both $N_{n-1}(A_q, \EE_m^{n-1}(p - q + 1))$ and $A_{q - 1}$.

For the first part, let $v \in N_{n-1}(A_q, \EE_m^{n-1}(p - q + 1))$ be arbitrary. Then there is $u \in A_q$ such that $uv \in E(\ZZ_m^{n-1})$. This means that $(u, q)(v, q) \in E(\ZZ_m^n)$, so $(v, q) \in N_n(A)$. Furthermore, $\cls(v) = p - q + 1$ implies $\cls((v, q)) = \cls(v) + q = p + 1$, so $(v, q) \in N_n(A, \EE_m^n(p + 1))$ and $v \in N_n(A, \EE_m^n(p + 1))_q$ follows.

For the second part, let $v \in A_{q - 1}$ be arbitrary. Then $(v, q - 1) \in A$, so $(v, q) \in N_n(A)$. Furthermore, $\cls((v, q - 1)) = p$ because of $(v, q - 1) \in A \subset \EE_m^n(p)$ implies $\cls((v, q)) = \cls((v, q - 1)) + 1 = p + 1$, so $(v, q) \in N_n(A, \EE_m^n(p + 1))$ and $v \in N_n(A, \EE_m^n(p + 1))_q$ follows.
\end{proof}

We shall also use two further easy observations.

\begin{lemma}\label{lem: circular}
Let $m \in \NN$ and $\delta > 0$. Suppose $\alpha_0, \ldots, \alpha_{m-1} \in \RR$ such that $\alpha_{q - 1} \leq \alpha_q + \delta$ for all $q \in \ZZ_m$. Then their mean $\alpha \coloneqq \sum_{q \in \ZZ_m} \alpha_q / m$ satisfies $\alpha_q \in [\alpha - \frac{m-1}2 \delta, \alpha + \frac{m-1}2 \delta]$ for all $q \in \ZZ_m$. The same conclusion also holds if $\alpha_{q + 1} \leq \alpha_q + \delta$ for all $q \in \ZZ_m$.
\end{lemma}

\begin{proof}
It suffices to only prove the case $\alpha_{q - 1} \leq \alpha_q + \delta$ as for $\alpha_{q + 1} \leq \alpha_q + \delta$, we can consider the sequence defined by $\alpha'_q \coloneqq \alpha_{-q}$ instead. So let $q \in \ZZ_m$ be arbitrary and observe that inductively, $\alpha_{q-r} \leq \alpha_q + r\delta$ holds for all $r \geq 0$ (reading the indices of $\alpha_i$'s modulo $m$). Now calculate
\[
	\alpha
	= \sum_{\tilde q \in \ZZ_m} \frac {\alpha_{\tilde q}}m
	= \sum_{r = 0}^{m-1} \frac {\alpha_{q-r}}m
	\leq \sum_{r = 0}^{m-1} \frac {\alpha_q + r\delta}m
	= \alpha_q + \frac \delta m \sum_{r = 0}^{m-1} r
	= \alpha_q + \frac {m-1}2 \delta
	\,.
\]

Reordering the assumption also guarantees that $\alpha_{q + 1} \geq \alpha_q - \delta$ for all $q \in \ZZ_m$, which inductively yields $\alpha_{q+r} \geq \alpha_q - r\delta$ for all $r \geq 0$. This allows us to obtain the inverse estimate $\alpha = \sum_{r = 0}^{m-1} \frac {\alpha_{q+r}}m \geq \alpha_q - \frac \delta m \sum_{r = 0}^{m-1} r = \alpha_q - \frac {m-1}2 \delta$ as well. Taken together, we have shown that $\abs {\alpha - \alpha_q} \leq \frac {m-1}2 \delta$ for all $q \in \ZZ_m$, which is equivalent to $\alpha_q \in [\alpha - \frac{m-1}2 \delta, \alpha + \frac{m-1}2 \delta]$. Since $q \in \ZZ_m$ was chosen arbitrary, this finishes the proof.
\end{proof}

\begin{lemma} \label{lem: concavity}
Let $a < b$ and $g \colon [a, b] \to \RR$ be concave. Then among all multisets $X$ with $m \in \NN$ elements and $\sum_{x \in X} x = m \frac {a + b}2$, the minimum value of $\sum_{x \in X} g(x)$ is achieved by $\ell \coloneqq \lfloor m/2 \rfloor$ copies of $a$, $\ell$ copies of $b$, and at most one copy of $\frac {a + b}2$.
\end{lemma}

\begin{proof}
Concavity implies that whenever $X$ contains two values $x, y$ with $a < x \leq y < b$, letting $\epsilon \coloneqq \min \{ x-a, b-y \}$ and replacing $x, y$ by $x - \epsilon$ and $y + \epsilon$ does not increase $\sum_{x \in X} g(x)$, while leaving $\sum_{x \in X} x$ unchanged. Inductively, we arrive at a minimizer $X$ that contains at most one value that is neither $a$ nor $b$. Straightforward calculation then shows that it must be the one claimed in the statement.
\end{proof}

We are now ready to prove the isoperimetric inequality. It verifies that when we consider a set $A$ in one partition class $\EE_m^n(p)$ and its neighbours in one of the adjacent partition classes, then there are at least $\frac {1 - \alpha}{C_m \sqrt n} \abs A$ additional neighbours apart from the trivial $\abs A$ many, where $\alpha \coloneqq \abs A / m^{n-1}$ is the relative size of $A \subset \EE_m^n(p)$ and the constant $C_m \coloneqq \sqrt{\ell^3m}$ with $\ell \coloneqq \lfloor m/2 \rfloor$ does not depend on $A$ or $n$.

\begin{proof}[Proof of~\cref{thm: isoperimetry}]
By symmetry, it suffices to prove the statement for $N_n(A, \EE_m^n(p + 1))$. We proceed by induction on $n$. For $n = 1$, every class consists of a single vertex, so either $\abs A = \abs {N_1(A, \EE_m^n(p + 1))} = 0$ or $\abs A = 1$, $\abs{N_1(A, \EE_m^n(p + 1))} = 1$, and $\alpha = 1$. In both cases, the inequality holds. So let $n > 1$ and define $\delta \coloneqq m\alpha(1 - \alpha) / \sqrt{\ell^3mn}$. It is easy to see that $A_q \subset \EE_m^{n-1}(p - q)$ for every $q \in \ZZ_m$ because of $A \subset \EE_m^n(p)$. In order to apply the induction hypothesis to $A_0, \ldots, A_{m-1}$, let $\alpha_q \coloneqq \abs {A_q} / m^{n-2}$ and distinguish two cases:

\noindent\emph{Case 1:} There is $\tilde q \in \ZZ_m$ such that $\alpha_{\tilde q - 1} \geq \alpha_{\tilde q} + \delta$. 

We use \cref{lem: slice-est} and the induction hypothesis to obtain the estimate
\[
	\abs {N_n(A, \EE_m^n(p + 1))_q} 
	\geq \abs{N_{n-1}(A_q, \EE_m^{n-1}(p - q + 1))} 
	\geq \abs {A_q}
\]
for all $q \in \ZZ_m \setminus \{ \tilde q \}$. For $\tilde q$ itself, we calculate
\[
	\abs {N_n(A, \EE_m^n(p + 1))_{\tilde q}}
	\geq \abs {A_{\tilde q - 1}}
	= \alpha_{\tilde q - 1} m^{n-2}
	\geq (\alpha_{\tilde q} + \delta)m^{n-2}
	= \abs {A_{\tilde q}} + \delta m^{n-2}
	\,.
\]
Summing up all these inequalities and plugging in the definition of $\delta$, we obtain as desired
\begin{align*}
	\abs {N_n(A, \EE_m^n(p + 1))}
	= \sum_{q \in \ZZ_m} \abs {N_n(A, \EE_m^n(p + 1))_q}
	&\geq \sum_{q \in \ZZ_m} \abs {A_q} + \frac {m\alpha (1 - \alpha)}{\sqrt{\ell^3mn}} m^{n-2} \\
	&= \abs A \left( 1 + \frac {1-\alpha}{\sqrt{\ell^3mn}} \right)
	\,.
\end{align*}

\noindent\emph{Case 2:} We have $\alpha_{q - 1} < \alpha_q + \delta$ for all $q \in \ZZ_m$. 

Here, \cref{lem: circular} implies that all $\alpha_q$ are within at most $\ell \delta$ of their mean, which is precisely $\sum_{q \in \ZZ_m} \alpha_q / m = \sum_{q \in \ZZ_m} \abs {A_q} / m^{n-1} = \abs A / m^{n-1} = \alpha$. So, $\alpha_q \in [\alpha - \ell \delta, \alpha + \ell \delta]$ holds for all $q \in \ZZ_m$. According to \cref{lem: slice-est} and the induction hypothesis, we can now bound
\begin{align}\label{eq: case-2}
	\abs {N_n(A, \EE_m^n(p + 1))}
	= \sum_{q \in \ZZ_m} \abs {N_n(A, \EE_m^n(p + 1))_q}
	&\geq \sum_{q \in \ZZ_m} \abs {N_{n-1}(A_q, \EE_m^{n-1}(p - q + 1))} \notag\\
	&\geq \sum_{q \in \ZZ_m} \abs {A_q} \left( 1 + \frac {1 - \alpha_q}{\sqrt{\ell^3m(n-1)}} \right) \notag\\
	&= \abs A + \frac {m^{n-2}}{\sqrt{\ell^3m(n-1)}} \sum_{q \in \ZZ_m} \alpha_q(1 - \alpha_q)
	\,.
\end{align}

Defining the function $g(\alpha) \coloneqq \alpha(1 - \alpha)$ on the interval $[\alpha - \ell \delta, \alpha + \ell \delta]$, we have to minimize the sum $\sum_{q \in \ZZ_m} g(\alpha_q)$ subject to the condition $\sum_{q \in \ZZ_m} \alpha_q = m \alpha$. As $g$ is concave, \cref{lem: concavity} guarantees that choosing $\alpha_0, \ldots, \alpha_{m-1}$ as $\ell$ copies of $\alpha - \ell \delta$, $\ell$ copies of $\alpha + \ell \delta$, and one copy of $\alpha$ itself yields the minimal value, which we calculate as
\[
	\sum_{q \in \ZZ_m} g(\alpha_q)
	\geq \ell g(\alpha - \ell \delta) + g(\alpha) + \ell g(\alpha + \ell \delta)
	= mg(\alpha) - 2\ell^3\delta^2
	\,.
\]
In order to determine the relative significance of the error term $-2\ell^3\delta^2$, we continue as follows:
\[
	\frac {mg(\alpha) - 2\ell^3\delta^2}{mg(\alpha)}
	= 1 - \frac {2\ell^3m^2\alpha^2(1 - \alpha)^2}{mg(\alpha)2\ell^3mn}
	= 1 - \frac {2g(\alpha)}{n}
	\geq 1 - \frac 1{2n}
	\geq \sqrt {1 - \frac 1n}
	= \sqrt { \frac {n-1}n }
	\,.
\]
Here, the two inequalities follow from the fact that the maximum of $g$ on $[0, 1] \ni \alpha$ is $g(1/2) = 1/4$ as well as the fact that $(1 - 1/(2n))^2 = 1 - 1/n + 1/(4n^2) \geq 1 - 1/n$. Finally, we plug this into inequality (\ref{eq: case-2}) to obtain as desired
\begin{align*}
	\abs {N_n(A, \EE_m^n(p + 1))}
	&\overset {\mathclap{(\ref{eq: case-2})}} \geq \abs A + \frac {m^{n-2}}{\sqrt{\ell^3m(n-1)}} \sum_{q \in \ZZ_m} g(\alpha_q) \\
	&\geq \abs A + \frac {m^{n-2} \cdot mg(\alpha)}{\sqrt{\ell^3m(n-1)}} \sqrt { \frac {n-1}n } \\
	&= \abs A \left( 1 + \frac {1 - \alpha}{\sqrt{\ell^3mn}} \right)
	\,.
\end{align*}
This concludes the proof.
\end{proof}

Obviously, this also establishes a lower bound on the total number of neighbours of $A$, irrespective of their class, as applying \cref{thm: isoperimetry} to both $p+1$ and $p-1$ yields double the bound. More generally, we can choose $A$ as an arbitrary subset of pairwise non-adjacent partition classes, which are automatically independent. This leads to the following corollary.

\begin{corollary}\label{cor: isoperimetry-more-classes}
Let $P \in \II(\ZZ_m^1)$ be an independent set in the $m$-cycle and $A \subset \bigcup_{p \in P} \EE_m^n(p)$. Setting $\alpha_p \coloneqq \abs {A \cap \EE_m^n(p)} / m^{n-1}$, we have 
\[
	\abs{N_{\ZZ_m^n}(A)} 
	\geq \abs A + \max_{p \in P} \abs {A \cap \EE_m^n(p)} + \frac {m^{n-1}}{\sqrt{\ell^3mn}} \sum_{p \in P} \alpha_p (1 - \alpha_p)
	\,.
\]
\end{corollary}

\begin{proof}
It is easy to see that $\abs {N_{\ZZ_m^1}(P)} \geq \abs P + 1$. Now let $Q \coloneqq N_{\ZZ_m^1}(P)$ and for each $q \in Q$ choose $p_q \in N_{\ZZ_m^1}(q, P)$. This can obviously be done such that $\{ p_q \colon q \in Q \} = P$, choosing the $p \in P$ with maximal $\alpha_p$ twice. We then apply \cref{thm: isoperimetry} to find that
\begin{align*}
	\abs{N_{\ZZ_m^n}(A \cap \EE_m^n(p_q), \EE_m^n(q))} 
	&\geq \abs {A \cap \EE_m^n(p_q)} \left( 1 + \frac {1-\alpha_{p_q}}{\sqrt{\ell^3mn}} \right) \\
	&= \abs {A \cap \EE_m^n(p_q)} + \frac {m^{n-1}}{\sqrt{\ell^3mn}} \alpha_{p_q}(1-\alpha_{p_q})
\end{align*}
for every $q \in Q$. As these sets are disjoint, adding up the inequalities yields
\begin{align} \label{eq: two-sums}
	\abs{N_{\ZZ_m^n}(A)}
	&\geq \sum_{q \in Q} \abs{N_{\ZZ_m^n}(A \cap \EE_m^n(p_q), \EE_m^n(q))} \notag\\
	&\geq \sum_{q \in Q} \abs {A \cap \EE_m^n(p_q)} + \frac {m^{n-1}}{\sqrt{\ell^3mn}} \sum_{q \in Q} \alpha_{p_q}(1-\alpha_{p_q})
	\,.
\end{align}
We recall that the $p_q$ were chosen in a way that guarantees that every $p \in P$ is chosen at least once. This means that $\sum_{q \in Q} \alpha_{p_q}(1 - \alpha_{p_q}) \geq \sum_{p \in P} \alpha_p(1 - \alpha_p)$. Moreover, the $p \in P$ with largest $\alpha_p$ is chosen twice. This means that $\sum_{q \in Q} \abs {A \cap \EE_m^n(p_q)}$ contains $\max_{p \in P} \abs {A \cap \EE_m^n(p)}$ twice, and so $\sum_{q \in Q} \abs {A \cap \EE_m^n(p_q)} \geq \abs A + \max_{p \in P} \abs {A \cap \EE_m^n(p)}$. Plugging both observations into (\ref{eq: two-sums}) yields the desired statement.
\end{proof}

In order to see how this is different from the isoperimetric inequality needed to deduce an upper bound on $\abs {\II(\ZZ_m^n)}$, consider the arguably easiest case $m = 3$. 
While \cref{thm: isoperimetry} examines sets $A$ in, say, $\EE_3^n(1)$, we would actually need to consider independent sets $A \subset \overline {I_0} = \EE_3^n(1) \cup \EE_3^n(2)$ and maintain an isoperimetric inequality of the form
\[
	\abs {N_{\ZZ_3^n}(A, \EE_3^n(0))}
	\geq \abs A \big( 1 + f(n, \abs A / \abs {\overline {I_0}}) \big)
	\,.
\]

Using the independence of $A \subset \overline {I_0}$, one easily observes that $\abs A / \abs {\overline {I_0}} \leq 1/2$ and might thus be tempted to hope that $f(n, \abs A / \abs {\overline {I_0}})$ depends on $\abs A / \abs {\overline {I_0}}$ only by involving a factor of $1 - 2 \abs A / \abs {\overline {I_0}}$. 
There is, however, the following counterexample: 
Let $A \coloneqq \overline {I_0} \cap I_1$ be the intersection of $\overline {I_0}$ with a different maximum independent set, for example $I_1 \coloneqq \iota_n(0, -1, 1, \ldots, 1)$. 
Then $\abs A / \abs {\overline {I_0}} = 1/3$ by the argument of \cref{lem: intersection maxindepsets}, but the neighbourhood of $A$ in $I_0$ does not include $I_0 \cap I_1$, so $\abs {N_{\ZZ_3^n}(A, I_0)} = 2 \cdot 3^{n-2} = \abs A$.
Yet, we still believe there is a constant $C$ such that every independent set $A \subset \overline {I_0}$ satisfies
\[
	\abs {N_{\ZZ_3^n}(A, \EE_3^n(0))}
	\geq \abs A \left( 1 + \frac {1 - 3 \abs A / \abs {\overline {I_0}}}{C \sqrt n} \right)
	\,.
\]
Unfortunately, we did not succeed in proving such a statement with the approach outlined in \cref{thm: isoperimetry}.

\section{Concluding remarks}
In this paper, we make progress on a question of Jenssen and Keevash~\cite{JK20} about the number of independent sets in Cartesian powers of the triangle.
We elaborate on several properties which illustrate that estimating this number may be much harder than the bipartite cases (including the hypercube) that have been considered so far;
one reason being the more complex isoperimetric inequality that is needed.

We establish in Theorem~\ref{thm: lowerboundK3} a lower bound on the number of independent sets in~$\ZZ_3^n$, which we conjecture to be asymptotically tight.
Clearly, it would be desirable to prove that this bound is indeed tight, but even finding an isoperimetric inequality as described after Theorem~\ref{thm: isoperimetry} would be interesting.

For Cartesian powers of larger odd cycles, we provide a less precise lower bound on the number of independent sets.
Moreover, we show how to approach this question with the cluster expansion method by calculating initial terms for $\ZZ_5^n$ and $\ZZ_7^n$.
More precise asymptotics and further progress towards an upper bound would again be highly desirable.

\section*{Acknowledgements}
We would like to thank Matthew Jenssen for many valuable ideas and discussions as well as for introducing us to the cluster expansion method.

\bibliographystyle{amsplain}
\bibliography{indepsets.bib}

\end{document}